\documentclass[10pt]{article}
\font\smallit=cmti10
\font\smalltt=cmtt10

\usepackage{amssymb,latexsym,amsmath,epsfig,amsthm,mathtools} 

\makeatletter

\renewcommand\section{\@startsection {section}{1}{\z@}
{-30pt \@plus -1ex \@minus -.2ex}
{2.3ex \@plus.2ex}
{\normalfont\normalsize\bfseries\boldmath}}

\renewcommand\subsection{\@startsection{subsection}{2}{\z@}
{-3.25ex\@plus -1ex \@minus -.2ex}
{1.5ex \@plus .2ex}
{\normalfont\normalsize\bfseries\boldmath}}

\renewcommand{\@seccntformat}[1]{\csname the#1\endcsname. }

\makeatother

\newtheorem{theorem}{Theorem}
\newtheorem{lemma}{Lemma}

\newtheorem{corollary}{Corollary}

\theoremstyle{definition}

\newtheorem{remark}{Remark}
\newtheorem{example}{Example}

\begin{document}

\begin{center}
\uppercase{\bf \boldmath A Dilogarithm Series and a Generalization of a Theorem~of Bridgeman}
\vskip 20pt
{\bf Chance Sanford}\\
{\tt sanfordchance@gmail.com}\\
\vskip 10pt
\end{center}
\vskip 20pt

\centerline{\smallit Received: , Revised: , Accepted: , Published: } 
\vskip 30pt

\centerline{\bf Abstract}
\noindent
Using Abel’s five-term relation, we derive a new two-parameter series identity for the Rogers dilogarithm. By specializing this identity, we obtain dilogarithm series involving Lucas sequences. These results generalize certain series identities of Bridgeman related to solutions of Pell’s equations, which were obtained via a completely different approach.

\pagestyle{myheadings}
\markright{\smalltt INTEGERS: 25 (2025)\hfill}
\thispagestyle{empty}
\baselineskip=12.875pt
\vskip 30pt

\section{Introduction}
Let $\mathrm{Li}_2(x)$ denote the dilogarithm, which we define for $0 \leq x \leq 1$ as
\begin{equation*}
    \mathrm{Li}_2(x) := \sum_{n=1}^{\infty}\frac{x^n}{n^2} = -\int_{0}^{x}\frac{\log(1-t)}{t}\mathrm{d}t.
\end{equation*}
The subject of this paper is a variant of $\mathrm{Li}_2(x)$, called the Rogers dilogarithm.  Denoted by $\mathcal{L}(x)$, the Rogers dilogarithm is defined for $0 < x < 1$ by 
\begin{align*}
\mathcal{L}(x) &:= \mathrm{Li}_2(x) + \frac{1}{2}\log(x)\log(1-x) \\
&= -\frac{1}{2}\int_{0}^{x}\left(\frac{\log(1-t)}{t}+\frac{\log(t)}{1-t}\right)\mathrm{d}t,
\end{align*}
and then extended to $0 \leq x \leq 1$ by setting $\mathcal{L}(0) := 0$ and $\mathcal{L}(1) := \pi^2/6$.  

Although not as widely known as other transcendental functions, the dilogarithm and its variants appear in a surprising number of contexts, from number theory and hyperbolic geometry to conformal field theory in mathematical physics.  A good introduction to the dilogarithm and its many applications can be found in the survey by Zagier \cite{zagier2007dilogarithm} and the first section of the article by Krillov \cite{kirillov1995dilogarithm}.

One interesting feature of the dilogarithm is that it satisfies a seemingly endless array of functional equations.  Of particular importance to us are the reflection identity,
\begin{equation}
    \mathcal{L}(x) + \mathcal{L}(1-x) = \frac{\pi^2}{6}, \label{eq:two_term}
\end{equation}
valid for $ 0 \leq x \leq 1$ and Abel's five-term relation,
\begin{equation}
    \mathcal{L}(x) + \mathcal{L}(y) = \mathcal{L}(xy) + \mathcal{L}\left(\frac{x(1-y)}{1-xy}\right) + \mathcal{L}\left(\frac{y(1-x)}{1-xy}\right), \label{eq:abel}
\end{equation}
which holds for all $0 < x,y < 1$.  The following special values may be deduced from these relations and the definition given above: 
\begin{equation*}
    \mathcal{L}\left(0\right) = 0, \quad \mathcal{L}\left(1\right) = \frac{\pi^2}{6}, \quad \mathcal{L}\left(\frac{1}{2}\right) = \frac{\pi^2}{12},\quad \mathcal{L}\left(\frac{1}{\varphi}\right) = \frac{\pi^2}{10}, \quad \mathcal{L}\left(\frac{1}{\varphi^2}\right) = \frac{\pi^2}{15},
\end{equation*}
where $\varphi = (1+\sqrt{5})/2$ is the golden ratio.  Curiously, these five values represent the only known instances where the Rogers dilogarithm defined on the unit interval may be evaluated in closed-form.    

The Rogers dilogarithm also possesses a number of beautiful series identities, of which the Richmond-Szekeres identity \cite{Richmond_Szekeres_1981} is the prototypical example: 
\begin{equation*}
    \sum_{n=2}^{\infty}\mathcal{L}\left(\frac{1}{n^2}\right) = \frac{\pi^2}{6}.
\end{equation*}
Hoorfar and Qi \cite{hoorfar2009sums} set out to generalize the Richmond-Szekeres identity, proving a number of new series identities by utilizing Abel's five-term functional equation for the Rogers dilogarithm.  A representative example from that work is the following theorem.
\begin{theorem}[Hoorfar and Qi \cite{hoorfar2009sums}]
For $p,q \in \mathbb{N}$ and $0 < \theta, \beta < 1$,
\begin{multline*}
    \sum_{n=0}^{\infty}\mathcal{L}\left(\frac{\beta(1-\theta^p)(1-\theta^q)\theta^n}{(1-\beta\theta^{n+p})(1-\beta\theta^{n+q})}\right) = \sum_{n=0}^{q-1}\mathcal{L}\left(\frac{\beta(1-\theta^p)\theta^n}{1-\beta\theta^{n+p}}\right) \\ + \sum_{n=0}^{p-1}\mathcal{L}\left(\frac{1-\theta^q}{1-\beta\theta^{n+q}}\right) - p \mathcal{L}(1-\theta^q).
\end{multline*}
\end{theorem}
Using a different approach, Bridgeman established a beautiful connection between solutions of Pell's equations, their associated continued fractions, and series involving the Rogers dilogarithm \cite{bridgeman2021dilogarithm}.  These results were obtained using tools from hyperbolic geometry, which were developed in Bridgeman's previous work \cite{bridgeman2011orthospectra}.  Very recently, Bridgeman's ideas were expanded upon by Jaipong et al. \cite{jaipong2023}, who discovered many more series identities for the Rogers dilogarithm.  

Although this paper does not utilize any results from hyperbolic geometry, in order to give the reader an idea of the approach taken by Bridgeman \cite{bridgeman2021dilogarithm} and Jaipong et al. \cite{jaipong2023}, we state Bridgeman's \textit{orthospectrum identity} \cite{bridgeman2011orthospectra}, which was the starting point for both sets of authors.  First, we recall that a hyperbolic surface is a complete, two-dimensional Riemannian manifold equipped with a metric of constant negative curvature.  Roughly speaking, the orthospectrum identity relates the lengths of certain geodesic arcs on a hyperbolic surface, called orthogeodesics, to particular properties of that surface.  More precisely, the orthospectrum identity is given by the following statement.
\begin{theorem}[Bridgeman \cite{bridgeman2011orthospectra}]
    Let $S$ denote a hyperbolic surface of finite area with totally geodesic boundary $\partial S \neq \emptyset$, and $N(S)$ boundary cusps.  In addition, let $O(S)$ be the set of orthogeodesics in $S$, that is, the set of geodesic arcs with endpoints on $\partial S$ and perpendicular to $\partial S$ at both ends. Then 
    \begin{equation*}
        \sum_{\alpha \in O(S)}\mathcal{L}\left(\frac{1}{\cosh^2(\ell(\alpha)/2)}\right) = -\frac{\pi^2}{6} \left(6 \chi(S) + N(S)\right)
    \end{equation*}
    where $\ell(\alpha)$ is the length of $\alpha$, and $\chi (S)$, the Euler characteristic of $S$, satisfies $\chi(S) = -\mathrm{Area}(S)/2\pi$.
\end{theorem}
By applying the orthospectrum identity to various hyperbolic surfaces, Bridgeman \cite{bridgeman2021dilogarithm} and Jaipong et al. \cite{jaipong2023} were able to derive dilogarithm series involving certain sequences.  The results in Bridgeman's paper \cite{bridgeman2021dilogarithm} were expressed in terms of the solutions to Pell's equations $x^2-ny^2=\pm 1$.  Bridgeman identifies a solution $(a,b) \in \mathbb{Q}^2$ to one of Pell's equations, with the quadratic irrational $u := a+b\sqrt{n}$.  Depending on whether $(a,b)$ is a solution to $x^2-ny^2= 1$ or $x^2-ny^2=-1$, $u$ is said to be a positive or negative solution, respectively.  One of Bridgeman's main theorems was the following result, which provides a series expression for $\mathcal{L}(1/u^2)$.
\begin{theorem}[Bridgeman \cite{bridgeman2021dilogarithm}]\label{thm:bridgeman}
    Let $u = a+b\sqrt{n} \in \mathbb{Q}[\sqrt{n}]$, with $a,b > 0$, satisfy Pell's equation. In addition, let $u^k = a_k + b_k \sqrt{n}$.  If $u$ is a positive solution, then
    \begin{equation}
        \mathcal{L}\left(\frac{1}{u^2}\right) = \sum_{k=2}^{\infty}\mathcal{L}\left(\frac{b^2}{b^2_k}\right). \label{eq:bridgeman_pos}
    \end{equation}
    Further, if $u \in \mathbb{Z}[\sqrt{n}]$ then $b_k/b \in \mathbb{Z}$ for all $k$.  If $u$ is a negative solution, then
    \begin{equation}
        \mathcal{L}\left(\frac{1}{u^2}\right) = \sum_{k=1}^{\infty}\mathcal{L}\left(\frac{a^2}{n b^2_{2k}}\right) + \sum_{k=1}^{\infty}\mathcal{L}\left(\frac{a^2}{a^2_{2k+1}}\right). \label{eq:bridgeman_neg}
    \end{equation}
    Further, if $u \in \mathbb{Z}[\sqrt{n}]$ then $b_{2k}/a, a_{2k+1}/a \in \mathbb{Z}$ for all $k$.
\end{theorem}
The aim of this paper is to explore an alternate approach to dilogarithm identities involving integer sequences, and in particular generalize Theorem~\ref{thm:bridgeman}.  Instead of using the ideas developed by Bridgeman and Jaipong et al., we take the approach of Hoorfar and Qi \cite{hoorfar2009sums}, relying only on elementary manipulations of Abel's functional equation.  Our results begin from a general two-parameter series, which upon specialization allows us to generalize Bridgeman's theorem.  
\section{The Principle Identity}
The principle result of this paper is the following theorem.
\begin{theorem}\label{thm:main}
    Let $a,b \in (0,1)$ be real numbers such that $a \neq b$.  Then
\begin{equation*}
    \sum_{n=0}^{\infty}\mathcal{L}\left(\frac{ab(a-b)^2(1-a)^n(1-b)^n}{(a(1-b)^{n+1}-b(1-a)^{n+1})^2}\right) = \mathcal{L}(a) + \mathcal{L}(b) - \mathcal{L}\left(\frac{\vert{a-b}\vert}{1-\min(a,b)}\right).
\end{equation*}
\end{theorem}
Before proving Theorem \ref{thm:main}, we must first establish a preliminary lemma.
\begin{lemma}\label{lemma:main}
    Let the sequences $(x_n)_{n\geq 0}$ and $(y_n)_{n\geq 0}$ be defined by
    \begin{equation*}
        x_{n} := \frac{a(a-b)(1-b)^n }{a(1-b)^{n+1} - b(1-a)^{n+1}} \quad \text{and} \quad y_{n} := \frac{b(a-b)(1-a)^n }{a(1-b)^{n+1} - b(1-a)^{n+1}}.
    \end{equation*}
    The terms $x_n$ and $y_n$ satisfy the coupled recurrence relations
    \begin{equation*}
        x_{n+1} = \frac{x_n(1-y_n)}{1-x_n y_n} \quad \text{and} \quad  y_{n+1} = \frac{y_n(1-x_n)}{1-x_n y_n}.
    \end{equation*}
\end{lemma}
\begin{proof}
    For convenience, let us first define the sequence $(D_{n}(a,b))_{n\geq 0}$ as
    \begin{equation*}
        D_{n}(a,b) := \frac{a(1-b)^{n+1} - b(1-a)^{n+1}}{a-b}.
    \end{equation*}
    Then the sequences $(x_{n})$ and $(y_{n})$ may be expressed in terms of $D_{n}(a,b)$ by writing
    \begin{equation*}
        x_{n} = \frac{a(1-b)^n}{D_n(a,b)} \quad \text{and} \quad  y_{n} = \frac{b(1-a)^n}{D_{n}(a,b)}.
    \end{equation*}
    Using these expressions we find that
    \begin{equation}
        \frac{x_n(1-y_n)}{1-x_n y_n} = \frac{a(1-b)^n(D_n(a,b)-b(1-a)^n)}{D^2_n(a,b) - ab(1-a)^n(1-b)^n}, \label{eq:lemma_proof1}
    \end{equation}
    and 
    \begin{equation}
        \frac{y_n(1-x_n)}{1-x_n y_n} = \frac{b(1-a)^n(D_n(a,b)-a(1-b)^n)}{D^2_n(a,b) - ab(1-a)^n(1-b)^n}. \label{eq:lemma_proof2}
    \end{equation}
    It may be shown by direct calculation that $D_{n}(a,b)$ possesses the Cassini-like identity
    \begin{equation*}
        D^2_n(a,b) -  D_{n-1}(a,b)D_{n+1}(a,b) = ab(1-a)^n(1-b)^n
    \end{equation*}
    as well as the identities 
    \begin{equation*}
        D_n(a,b)- (1-b)D_{n-1}(a,b)= b(1-a)^n,
    \end{equation*}
    and 
    \begin{equation*}
        D_n(a,b)-(1-a)D_{n-1}(a,b) =  a(1-b)^n.
    \end{equation*}
    Therefore, Equations \eqref{eq:lemma_proof1} and \eqref{eq:lemma_proof2} may be written as
    \begin{equation*}
        \frac{x_n(1-y_n)}{1-x_n y_n} = \frac{a(1-b)^{n+1}D_{n-1}(a,b)}{D_{n-1}(a,b)D_{n+1}(a,b)} = \frac{a(1-b)^{n+1}}{D_{n+1}(a,b)} = x_{n+1}
    \end{equation*}
    and
    \begin{equation*}
        \frac{y_n(1-x_n)}{1-x_n y_n} = \frac{b(1-a)^{n+1}D_{n-1}(a,b)}{D_{n-1}(a,b)D_{n+1}(a,b)} = \frac{b(1-a)^{n+1}}{D_{n+1}(a,b)} = y_{n+1},
    \end{equation*}
    which is exactly what we hoped to show.  This completes the proof of Lemma \ref{lemma:main}.
\end{proof}
With Lemma \ref{lemma:main} established, we may now move on to the proof of Theorem~\ref{thm:main}.
\begin{proof}[Proof of Theorem~\ref{thm:main}]
    As we have in Lemma \ref{lemma:main}, define the sequences $(x_n)_{n\geq 0}$ and $(y_n)_{n\geq 0}$ by
    \begin{equation*}
        x_{n} := \frac{a(a-b)(1-b)^n }{a(1-b)^{n+1} - b(1-a)^{n+1}} \quad \text{and} \quad y_{n} := \frac{b(a-b)(1-a)^n }{a(1-b)^{n+1} - b(1-a)^{n+1}}.
    \end{equation*}
    Lemma \ref{lemma:main} showed that
    \begin{equation*}
        x_{n+1} = \frac{x_n(1-y_n)}{1-x_n y_n} \quad \text{and} \quad  y_{n+1} = \frac{y_n(1-x_n)}{1-x_n y_n}.
    \end{equation*}
    From these expressions, it is not difficult to verify by induction that $0 < x_n ,y_n < 1$ when $a,b \in (0,1)$.  Thus, with Lemma \ref{lemma:main} in mind, by setting $x = x_n$ and $y= y_n$ in Abel's functional Equation~\eqref{eq:abel}, we have
    \begin{align*}
    \mathcal{L}(x_n y_n) &= \mathcal{L}(x_n) + \mathcal{L}(y_n) - \mathcal{L}\left(\frac{x_n(1-y_n)}{1-x_n y_n}\right) - \mathcal{L}\left(\frac{y_n(1-x_n)}{1-x_n y_n}\right) \\
    &= \mathcal{L}(x_n) + \mathcal{L}(y_n) - \mathcal{L}(x_{n+1}) - \mathcal{L}(y_{n+1}).
\end{align*}
Therefore, in summing each side we find that
\begin{equation}
   \sum_{n=0}^{\infty}\mathcal{L}(x_n y_n) = \mathcal{L}(x_0) + \mathcal{L}(y_0) - \lim_{n\to \infty}\mathcal{L}(x_n) - \lim_{n\to \infty}\mathcal{L}(y_n). \label{eq:main_telescopic_series}
\end{equation}
In order to evaluate the limit of the sequences $(x_n)$ and $(y_n)$, note that we may write
\begin{equation*}
    x_{n} = \frac{a(a-b)}{a(1-b) - b(1-a)\left(\frac{1-a}{1-b}\right)^n} \quad \text{and} \quad y_{n} = \frac{b(a-b)}{a(1-b)\left(\frac{1-b}{1-a}\right)^n - b(1-a)}.
\end{equation*}
Thus, the limit of $x_n$ and $y_n$ is determined by $\lim_{n \to \infty} \left(\frac{1-a}{1-b}\right)^{\pm n}$.  Since $ \left(\frac{1-a}{1-b}\right)^{n}$ vanishes in the limit when $1-a < 1 - b$, we arrive at the following evaluations:
\begin{equation*}
    \lim_{n\to\infty}x_n = \begin{cases}
        \frac{a-b}{1-b} & \text{if } a > b \\
        0 & \text{if } a < b
        \end{cases} \quad \text{and} \quad \lim_{n\to\infty}y_n = \begin{cases}
        0 & \text{if } a > b \\
        \frac{b-a}{1-a} & \text{if } a < b.
        \end{cases}
\end{equation*}
Consequently, Equation~\eqref{eq:main_telescopic_series} may be written as
    \begin{equation*}
     \sum_{n=0}^{\infty}\mathcal{L}\left(\frac{ab(a-b)^2(1-a)^n(1-b)^n}{(a(1-b)^{n+1}-b(1-a)^{n+1})^2}\right) = \mathcal{L}(a) + \mathcal{L}(b) - \begin{cases}
         \mathcal{L}\left(\frac{a-b}{1-b}\right) & \text{if } a > b \\
         \mathcal{L}\left(\frac{b-a}{1-a}\right) & \text{if } a < b.
     \end{cases}
     \end{equation*}
     Since this is equivalent to the result stated in Theorem \ref{thm:main}, our proof is complete.
\end{proof}
Before proceeding, let us outline the structure of the remainder of the paper.  In the following section, we establish our generalization of Theorem~\ref{thm:bridgeman} and provide a number of special cases.  In addition, we show that in our case, the two identities which comprise our generalization are more closely related than they might first appear.  Finally, in the last section we explain exactly how Bridgeman's Theorem~\ref{thm:bridgeman} relates to our generalization of it, which is not immediately clear since one is expressed in terms of solutions to Pell's equation while the other is expressed in terms of Lucas sequences.
\section{Series Involving Lucas Sequences}
The Lucas sequences $(U_n(P,Q))_{n\geq 0}$ and $(V_n(P,Q))_{n\geq 0}$ are second order linear recurrence sequences with constant coefficients.  They generalize the Fibonacci and Lucas numbers and possess many other well-known sequences as special cases.

The Lucas sequence $(U_n(P,Q))_{n\geq 0}$ is defined by the recurrence
\begin{align*}
    U_0(P,Q) &= 0, \\
    U_1(P,Q) &= 1, \\
    U_{n}(P,Q) &= P\cdot U_{n-1}(P,Q) - Q\cdot U_{n-2}(P,Q),
\end{align*}
while the sequence $(V_n(P,Q))_{n\geq 0}$ satisfies the same recurrence albeit with different initial values:
\begin{align*}
    V_0(P,Q) &= 2, \\
    V_1(P,Q) &= P, \\
    V_{n}(P,Q) &= P\cdot V_{n-1}(P,Q) - Q\cdot V_{n-2}(P,Q).
\end{align*}
The characteristic equation for the Lucas sequences is given by
\begin{equation*}
    x^2-Px+Q = 0,
\end{equation*}
with a discriminant of $D := P^2-4Q$ and roots
\begin{equation*}
    \alpha = \frac{P+\sqrt{D}}{2} \quad \text{and} \quad \beta = \frac{P- \sqrt{D}}{2}.
\end{equation*}
In addition, if $D \neq 0$ so that $\alpha$ and $\beta$ are distinct, one may show that
\begin{equation*}
    \alpha^n = \frac{V_n+U_n\sqrt{D}}{2} \quad \text{and} \quad \beta^n = \frac{V_n - U_n\sqrt{D}}{2}.
\end{equation*}
Therefore, $U_n(P,Q)$ and $V_n(P,Q)$ possess the Binet forms
\begin{equation*}
    U_n(P,Q) = \frac{\alpha^n - \beta^n}{\alpha - \beta} = \frac{\alpha^n - \beta^n}{\sqrt{D}} \quad \text{and} \quad V_n(P,Q) = \alpha^n + \beta^n.
\end{equation*}
In the remainder of this paper we will assume that $P$ and $Q$ are real numbers such that $P > 0, Q \neq 0$, and $D > 0$.

The starting point for the results in this section is the following corollary of Theorem~\ref{thm:main}.
\begin{corollary}\label{cor:main}
    If $a \in (0,1)$, then 
    \begin{equation*}
        \sum_{n=1}^{\infty}\mathcal{L}\left(\frac{a^2\left(\frac{1}{2}-\frac{a}{2}\right)^{n}\left(\frac{1}{2} + \frac{a}{2}\right)^{n}}{\left(\left(\frac{1}{2}+\frac{a}{2}\right)^{n+1}-\left(\frac{1}{2} - \frac{a}{2}\right)^{n+1}\right)^2}\right) = \mathcal{L}\left(\frac{1-a}{1+a}\right).
    \end{equation*}  
\end{corollary}
\begin{proof}
    If we let $a = \frac{1}{2}+\frac{a}{2}$ and $b = \frac{1}{2}-\frac{a}{2}$ in Theorem~\ref{thm:main} we find that
\begin{equation*}
    \sum_{n=0}^{\infty}\mathcal{L}\left(\frac{a^2(\frac{1}{2}-\frac{a}{2})^{n+1}(\frac{1}{2}+\frac{a}{2})^{n+1}}{((\frac{1}{2}+\frac{a}{2})^{n+2}-(\frac{1}{2}-\frac{a}{2})^{n+2})^2}\right) = \mathcal{L}\left(\frac{1}{2}+\frac{a}{2}\right) + \mathcal{L}\left(\frac{1}{2}-\frac{a}{2}\right) - \mathcal{L}\left(\frac{2a}{1+a}\right).
\end{equation*}
By re-indexing and using the two-term functional equation, Equation~\eqref{eq:two_term}, to simplify the right side, we obtain
\begin{equation*}
    \sum_{n=1}^{\infty}\mathcal{L}\left(\frac{a^2(\frac{1}{2}-\frac{a}{2})^{n}(\frac{1}{2}+\frac{a}{2})^{n}}{((\frac{1}{2}+\frac{a}{2})^{n+1}-(\frac{1}{2}-\frac{a}{2})^{n+1})^2}\right) = \frac{\pi^2}{6} - \mathcal{L}\left(\frac{2a}{1+a}\right).
\end{equation*}
Then, applying Equation~\eqref{eq:two_term} once more to the right side completes the proof.
\end{proof}
\begin{remark}
    Although the form given above is better suited for our purposes below, the series in Corollary 1 can be expressed more simply as
\begin{equation*}
    \sum_{n=1}^{\infty}\mathcal{L}\left(\frac{a^2\left(1-a^2\right)^{n}}{4\left(\left(1-a\right)^{n+1}-\left(1+a\right)^{n+1}\right)^2}\right) = \mathcal{L}\left(\frac{1-a}{1+a}\right).
\end{equation*}
\end{remark}
With Corollary \ref{cor:main} established, we are now able to present our generalization of Bridgeman's theorem.  
\begin{theorem}\label{thm:generalized_bridgeman}
    If $Q > 0$ and $k \in \mathbb{Z}_{>0}$, then
    \begin{equation}
        \sum_{n=1}^{\infty}\mathcal{L}\left(\frac{U^2_k Q^{kn}}{U^2_{k(n+1)}}\right) = \mathcal{L}\left(\frac{Q^k}{\alpha^{2k}}\right).\label{eq:lucas_series_pos}
    \end{equation}
    If $Q < 0$ and $k \in \mathbb{Z}_{>0}$ is an odd integer, then 
    \begin{equation}
       \sum_{n=1}^{\infty}\mathcal{L}\left(-\frac{V^2_k Q^{k(2n-1)}}{D U^2_{2kn} }\right) + \sum_{n=1}^{\infty}\mathcal{L}\left(\frac{V^2_k Q^{2kn}}{ V^2_{k(2n+1)}}\right) = \mathcal{L}\left(-\frac{Q^k}{\alpha^{2k}}\right).\label{eq:lucas_series_neg}
    \end{equation} 
\end{theorem}
\begin{proof}
To prove Equation~\eqref{eq:lucas_series_pos}, notice that since $Q > 0$ and $D = P^2 - 4Q > 0$, $\beta = (P-\sqrt{D})/2$ is also positive. 
 Therefore, the quotient
\begin{equation*}
    \frac{\sqrt{D} U_k}{V_k} = \frac{\alpha^k - \beta^k}{\alpha^k + \beta^k}
\end{equation*}
    is less than 1 for all $k \in \mathbb{Z}_{> 0}$.  Thus, we are permitted to set $a = U_k\sqrt{D}/V_k$ in Corollary \ref{cor:main}.  In doing so, we find that
    \begin{equation}
        \sum_{n=1}^{\infty}\mathcal{L}\left(\frac{U^2_k D\left(\frac{1}{2}-\frac{U_k\sqrt{D}}{2V_k}\right)^{n}\left(\frac{1}{2} + \frac{U_k\sqrt{D}}{2 V_k}\right)^{n}}{V^2_k\left(\left(\frac{1}{2}+\frac{U_k\sqrt{D}}{2 V_k}\right)^{n+1}-\left(\frac{1}{2} - \frac{U_k \sqrt{D}}{2 V_k}\right)^{n+1}\right)^2}\right) = \mathcal{L}\left(\frac{V_k-U_k \sqrt{D}}{V_k+U_k\sqrt{D}}\right).\label{eq:lucas_proof_1}
    \end{equation}
    Examining the left side of Equation \eqref{eq:lucas_proof_1} first, we see that the numerator of the dilogarithm's argument may be expressed as
    \begin{align*}
        U^2_k D\left(\tfrac{1}{2}-\tfrac{U_k\sqrt{D}}{2V_k}\right)^{n}\left(\tfrac{1}{2} + \tfrac{U_k\sqrt{D}}{2 V_k}\right)^{n} &= U^2_k D V^{-2n}_k\left(\tfrac{V_k-U_k\sqrt{D}}{2}\right)^{n}\left(\tfrac{V_k+U_k\sqrt{D}}{2}\right)^{n} \\
        &= U^2_k D V^{-2n}_k \beta^{kn} \alpha^{kn} \\
        &= U^2_k D V^{-2n}_k Q^{kn},
    \end{align*}
    while the denominator may be written as
    \begin{align*}
        \MoveEqLeft[12]  V^2_k\left(\left(\tfrac{1}{2}+\tfrac{U_k\sqrt{D}}{2 V_k}\right)^{n+1} -\left(\tfrac{1}{2} - \tfrac{U_k \sqrt{D}}{2  V_k}\right)^{n+1}\right)^2 \\ 
        ={}& V^{-2n}_k\left(\left(\tfrac{V_k + U_k\sqrt{D}}{2}\right)^{n+1}-\left(\tfrac{V_k - U_k \sqrt{D}}{2}\right)^{n+1}\right)^2 \\
        ={}& V^{-2n}_k \left(\alpha^{k(n+1)} - \beta^{k(n+1)}\right)^2.
    \end{align*}
    Therefore, the sum simplifies to
    \begin{equation*}
        \sum_{n=1}^{\infty}\mathcal{L}\left(\frac{U^2_k D Q^{kn}}{(\alpha^{k(n+1)} - \beta^{k(n+1)})^2}\right) =  \sum_{n=1}^{\infty}\mathcal{L}\left(\frac{U^2_k Q^{kn}}{U^2_{k(n+1)}}\right).
    \end{equation*}
    
    Turning our attention to the right side of Equation~\eqref{eq:lucas_proof_1}, we see that 
    \begin{align*}
        \mathcal{L}\left(\frac{V_k-U_k \sqrt{D}}{V_k+U_k\sqrt{D}}\right) &= \mathcal{L}\left(\frac{V^2_k-U^2_k D}{(V_k+U_k\sqrt{D})^2}\right) \\
        &= \mathcal{L}\left(\frac{4Q^k}{(V_k+U_k\sqrt{D})^2}\right) \\
        &= \mathcal{L}\left(\frac{Q^k}{\alpha^{2k}}\right).
    \end{align*}
    Thus, we have shown that
    \begin{equation*}
        \sum_{n=1}^{\infty}\mathcal{L}\left(\frac{U^2_k Q^{kn}}{U^2_{k(n+1)}}\right)  = \mathcal{L}\left(\frac{Q^k}{\alpha^{2k}}\right),
    \end{equation*}
   proving Equation~\eqref{eq:lucas_series_pos}.
   
   Moving on to Equation~\eqref{eq:lucas_series_neg}, since $Q < 0$ and $k$ is an odd integer, we see that $\beta^k < 0$.  Therefore, the quotient
\begin{equation*}
    \frac{V_k}{\sqrt{D} U_k} = \frac{\alpha^k + \beta^k}{\alpha^k - \beta^k}
\end{equation*}
is less than 1.  Setting $a = V_k/U_k \sqrt{D}$ in Corollary \ref{cor:main}, we find that
    \begin{equation*}
        \sum_{n=1}^{\infty}\mathcal{L}\left(\frac{V^2_k\left(\frac{1}{2}-\frac{V_k}{2U_k\sqrt{D}}\right)^{n}\left(\frac{1}{2} + \frac{V_k}{2U_k\sqrt{D}}\right)^{n}}{U^2_k D\left(\left(\frac{1}{2}+\frac{V_k}{2U_k\sqrt{D}}\right)^{n+1}-\left(\frac{1}{2} - \frac{V_k}{2U_k\sqrt{D}}\right)^{n+1}\right)^2}\right) = \mathcal{L}\left(\frac{U_k\sqrt{D}-V_k}{U_k\sqrt{D}+V_k}\right).
    \end{equation*}
    Once again, simplifying the numerator and denominator separately, the numerator is seen to be
    \begin{equation*}
        V^2_k\left(\tfrac{1}{2}-\tfrac{V_k}{2U_k\sqrt{D}}\right)^{n}\left(\tfrac{1}{2} + \tfrac{V_k}{2U_k\sqrt{D}}\right)^{n} = (-1)^n U^{-2n}_k D^{-n} V^2_k Q^{kn},
    \end{equation*}
    while the denominator admits the form
    \begin{multline*}
     U^2_k D\left(\left(\tfrac{1}{2}+\tfrac{V_k}{2U_k\sqrt{D}}\right)^{n+1}-\left(\tfrac{1}{2} - \tfrac{V_k}{2U_k\sqrt{D}}\right)^{n+1}\right)^2 \quad\quad \quad\quad \quad\quad \quad\quad {}\\
     \begin{aligned}
        &= U^{-2n}D^{-n}\left(\left(\tfrac{V_k+U_k\sqrt{D}}{2}\right)^{n+1}-(-1)^{n+1}\left(\tfrac{V_k-U_k\sqrt{D}}{2} \right)^{n+1}\right)^2 \\
        &= U^{-2n} D^{-n} (\alpha^{k(n+1)} + (-1)^{n} \beta^{k(n+1)})^2.
     \end{aligned}
    \end{multline*}
    Putting this all together yields 
    \begin{equation*}
        \sum_{n=1}^{\infty}\mathcal{L}\left(\frac{V^2_k (-1)^n Q^{kn}}{\left(\alpha^{k(n+1)}+(-1)^{n}\beta^{k(n+1)}\right)^2 }\right) = \mathcal{L}\left(\frac{U_k\sqrt{D}-V_k}{U_k\sqrt{D}+V_k}\right).
    \end{equation*}
    
    Next, by splitting the series into two sums, one over all even indices and the other over all odd indices, we find that
    \begin{multline*}
        \sum_{n=1}^{\infty}\mathcal{L}\left(\frac{V^2_k (-1)^n Q^{kn}}{\left(\alpha^{k(n+1)}+(-1)^{n}\beta^{k(n+1)}\right)^2 }\right) \\ = \sum_{n=1}^{\infty}\mathcal{L}\left(-\frac{V^2_k Q^{k(2n-1)}}{\left(\alpha^{2kn}-\beta^{2kn}\right)^2 }\right) + \sum_{n=1}^{\infty}\mathcal{L}\left(\frac{V^2_k Q^{2kn}}{\left(\alpha^{k(2n+1)}+\beta^{k(2n+1)}\right)^2 }\right).
    \end{multline*}
    Using the Binet formulas for $U_n$ and $V_n$, we arrive at
        \begin{multline*}
        \sum_{n=1}^{\infty}\mathcal{L}\left(\frac{V^2_k (-1)^n Q^{kn}}{\left(\alpha^{k(n+1)}+(-1)^{n}\beta^{k(n+1)}\right)^2 }\right) \\ = \sum_{n=1}^{\infty}\mathcal{L}\left(-\frac{V^2_k Q^{k(2n-1)}}{D U^2_{2kn} }\right) + \sum_{n=1}^{\infty}\mathcal{L}\left(\frac{V^2_k Q^{2kn}}{ V^2_{k(2n+1)}}\right).
    \end{multline*}
    Finally, the right-hand side may be simplified in the same way we proceeded in the proof of Equation~\eqref{eq:lucas_series_pos} to show that
    \begin{equation*}
        \mathcal{L}\left(\frac{U_k\sqrt{D}-V_k}{U_k\sqrt{D}+V_k}\right) = \mathcal{L}\left(-\frac{Q^k}{\alpha^{2k}}\right).
    \end{equation*}
    This completes the proof of Theorem~\ref{thm:generalized_bridgeman}.
\end{proof}
\begin{remark}
    When the parameters $P$ and $Q$ are integers, the Lucas sequence $U_n(P,Q)$ is a \textit{divisibility sequence}, in the sense that if $m$ divides $n$ then $U_m$ divides $U_n$.  Moreover, if $\gcd(P,Q) = 1$, then $U_n(P,Q)$ is a \textit{strong divisibility sequence}, which means that
    \begin{equation*}
        \gcd(U_m,U_n) = U_{\gcd(m,n)}.
    \end{equation*}
    Consequently, when $P,Q \in \mathbb{Z}$, the Lucas sequence $U^2_k$ fully divides $U^2_{k(n+1)}$ in the summand of Equation~\eqref{eq:lucas_series_pos}.
\end{remark}

We now present some specific examples of Theorem~\ref{thm:generalized_bridgeman}.  Many of our examples have already been exhibited in the previous works of Hoorfar and Qi \cite{hoorfar2009sums}, Bridgeman \cite{bridgeman2021dilogarithm}, and Jaipong et al. \cite{jaipong2023}.  Nevertheless, we restate them here as well as present some new identities, including the initial equations of Examples \ref{ex:chebyshev} and \ref{ex:repunits}, which generalize existing results. 
\begin{example}\label{ex:lucas_pos_fib}
If $P=3$ and $Q=1$, then $U_n(3,1) = F_{2n}$, where $F_{n}$ is the $n$-th Fibonacci number.  With these values of $P$ and $Q$, Equation~\eqref{eq:lucas_series_pos} yields
\begin{equation*}
    \sum_{n=2}^{\infty}\mathcal{L}\left(\frac{F^2_{2k}}{F^2_{2kn}}\right)  = \mathcal{L}\left(\frac{1}{\varphi^{4k}}\right),
\end{equation*}
which was given as an example in Jaipong et al. \cite{jaipong2023}.  In particular, if $k = 1$, then we have the nice identity
\begin{equation*}
    \sum_{n=2}^{\infty}\mathcal{L}\left(\frac{1}{F^2_{2n}}\right)  = \mathcal{L}\left(\frac{2}{7+3\sqrt{5}}\right).
\end{equation*}
\end{example}
\begin{example}\label{ex:chebyshev}
If $P=2x$ and $Q=1$, then $U_n(2x,1) = U_{n-1}(x)$, where $U_{n}(x)$ is the $n$-th Chebyshev polynomial of the second kind.  Equation~\eqref{eq:lucas_series_pos} in this case returns
\begin{equation*}
    \sum_{n=2}^{\infty}\mathcal{L}\left(\frac{U^2_{k-1}(x)}{U^2_{kn-1}(x)}\right)  = \mathcal{L}\left(\frac{1}{(x+\sqrt{x^2-1})^{2k}}\right).
\end{equation*}
Letting $x = \cosh(\theta)$ and $k=1$, the identity $U_{n-1}(\cosh(\theta)) = \sinh(n \theta) / \sinh(\theta)$ allows us to deduce the interesting series 
\begin{equation*}
    \sum_{n=2}^{\infty}\mathcal{L}\left(\frac{\sinh^2(\theta)}{\sinh^2(n \theta)}\right)  = \mathcal{L}\left(e^{-2\theta}\right).
\end{equation*}
This identity was first derived by Hoorfar and Qi \cite{hoorfar2009sums} and later rediscovered by Bridgeman \cite{bridgeman2021dilogarithm}.
\end{example}
\begin{example}\label{ex:repunits}
If $P=x+1$ and $Q=x$ with $x > 1$, then 
\begin{equation*}
    U_n(x+1,x) = \frac{x^n-1}{x-1} =1+x+x^2+\cdots + x^{n-1}.
\end{equation*}
Applying Equation \eqref{eq:lucas_series_pos}, we obtain
\begin{equation*}
    \sum_{n=1}^{\infty}\mathcal{L}\left(x^{kn}\left(\frac{x^k-1}{x^{k(n+1)}-1}\right)^2\right)  = \mathcal{L}\left(\frac{1}{x^k}\right).
\end{equation*}
Setting $k=1$ gives the special case
\begin{equation*}
    \sum_{n=1}^{\infty}\mathcal{L}\left(\left(\frac{x^{n/2}}{1+x+x^2+\cdots + x^{n}}\right)^2\right)  = \mathcal{L}\left(\frac{1}{x}\right),
\end{equation*}
which was given as an example by Jaipong et al. \cite{jaipong2023}.
\end{example}
\begin{example}\label{ex:lucas_neg_fib}
If $P=1, Q=-1$, and $k$ is an odd integer, then $U_n(1,-1) = F_n$ and $V_n(1,-1) = L_n$, where $L_{n}$ is the $n$-th Lucas number.  Equation~\eqref{eq:lucas_series_neg} shows that
\begin{equation*}
    \sum_{n=1}^{\infty}\mathcal{L}\left(\frac{L^2_k}{5 F^2_{2kn} }\right) + \sum_{n=1}^{\infty}\mathcal{L}\left(\frac{L^2_k}{L^2_{k(2n+1)} }\right) = \mathcal{L}\left(\frac{1}{\varphi^{2k}}\right).
\end{equation*}
Specifically, for $k = 1$ we have
\begin{equation*}
    \sum_{n=1}^{\infty}\mathcal{L}\left(\frac{1}{5 F^2_{2n} }\right) + \sum_{n=1}^{\infty}\mathcal{L}\left(\frac{1}{L^2_{2n+1} }\right) = \frac{\pi^2}{15}.
\end{equation*}
This example was first derived by Bridgeman \cite{bridgeman2021dilogarithm}. 
\end{example} 
\begin{example}
If $P=2, Q=-1$, and $k$ is an odd integer, then $U_n(2,-1) = P_n$ and $V_n(2,-1) = Q_n$, where $P_{n}$ is the $n$-th Pell number and $Q_{n}$ is the $n$-th Pell-Lucas number. With these values, Equation~\eqref{eq:lucas_series_neg} yields the identity
\begin{equation*}
    \sum_{n=1}^{\infty}\mathcal{L}\left(\frac{Q^2_k}{8 P^2_{2kn} }\right) + \sum_{n=1}^{\infty}\mathcal{L}\left(\frac{Q^2_k}{Q^2_{k(2n+1)} }\right) = \mathcal{L}\left(\frac{1}{(1+\sqrt{2})^{2k}}\right).
\end{equation*}
For $k = 1$, we have
\begin{equation*}
    \sum_{n=1}^{\infty}\mathcal{L}\left(\frac{1}{2 P^2_{2n} }\right) + \sum_{n=1}^{\infty}\mathcal{L}\left(\frac{4}{Q^2_{2n+1} }\right) = \mathcal{L}\left(\frac{1}{3+2\sqrt{2}}\right).
\end{equation*}
\end{example}
\begin{example}\label{ex:1,-3}
    Let $P=1$ and $Q = -3$ and set $A_n = U(1,-3)$ and ${B_n = V_n(1,-3)}$.  The first handful of terms of $(A_n)$ and $ (B_n)$ are given by
    \begin{equation}
        (A_n)_{n\geq 0} = (0,1,1,4,7,19,40,97,217,508,1159,\cdots)
    \end{equation}
    and 
    \begin{equation}
        (B_n)_{n\geq 0} = (2,1,7,10,31,61,154,337,799,1810,4207,\dots).
    \end{equation} 
    For odd $k$, Equation~\eqref{eq:lucas_series_neg} gives
    \begin{equation*}
        \sum_{n=1}^{\infty}\mathcal{L}\left(\frac{B^2_k 3^{k(2n-1)} }{13 A^2_{2kn} }\right) + \sum_{n=1}^{\infty}\mathcal{L}\left(\frac{B^2_k 3^{2kn}}{ B^2_{k(2n+1)} }\right) = \mathcal{L}\left(\left(\frac{6}{7+\sqrt{13}}\right)^k\right).
    \end{equation*}
    Letting $k=1$, we find that
    \begin{equation*}
        \sum_{n=1}^{\infty}\mathcal{L}\left(\frac{3^{2n-1} }{13 A^2_{2n} }\right) + \sum_{n=1}^{\infty}\mathcal{L}\left(\frac{3^{2n}}{ B^2_{2n+1} }\right) = \mathcal{L}\left(\frac{6}{7+\sqrt{13}}\right).
    \end{equation*}
\end{example}

Before moving on, we point out that since we have not required $P$ and $Q$ to take integer values, it can be shown that Equation~\eqref{eq:lucas_series_neg} of Theorem~\ref{thm:generalized_bridgeman} is actually a special case of Equation~\eqref{eq:lucas_series_pos} of the same theorem.  To see this, suppose that $Q < 0$ and set
    \begin{equation*}
        P': = \sqrt{P^2-4Q}, \quad \text{and} \quad Q' := -Q.
    \end{equation*}
    If we also let $D' := {P'}^2- 4Q'$, then
    \begin{equation*}
        \alpha' := \frac{P'+\sqrt{D'}}{2} = \frac{\sqrt{D}+P}{2} = \alpha,
    \end{equation*}
    and
    \begin{equation*}
        \beta' := \frac{P'-\sqrt{D'}}{2} = \frac{\sqrt{D}-P}{2} = -\beta.
    \end{equation*}
    Therefore, the Lucas sequence $U_n(P',Q')$ is given by
    \begin{equation*}
        U_n(P',Q') = \frac{{\alpha'}^n - {\beta'}^n}{\alpha' - \beta'} = \frac{\alpha^n-(-1)^n \beta^n}{P}.
    \end{equation*}
    This may be expressed in terms of the Lucas sequences as
    \begin{equation*}
         U_n(P',Q') = \begin{cases}
            \sqrt{D}U_n(P,Q)/P & \text{if } n \text{ even} \\
            V_n(P,Q)/P & \text{if } n \text{ odd}.
        \end{cases}
    \end{equation*}
    Similarly, the sequence $(V_n(P',Q'))$ is given by
    \begin{equation*}
         V_n(P',Q') = \begin{cases}
            V_n(P,Q)/P & \text{if } n \text{ even} \\
            \sqrt{D}U_n(P,Q)/P & \text{if } n \text{ odd}.
        \end{cases}
    \end{equation*}
    Employing this transformation in Equation~\eqref{eq:lucas_series_pos} when $k$ is odd, we obtain Equation~\eqref{eq:lucas_series_neg}. 
 To illustrate this, let us look at the following example which was exhibited in Jaipong et al. \cite{jaipong2023}.
    \begin{example}
If $P=\sqrt{5}$ and $Q=1$, then
\begin{equation*}
    U_n(\sqrt{5},1) = \begin{cases}
    \sqrt{5}F_{n} &\text{if } n \text{ even} \\
    L_{n} &\text{if } n \text{ odd}.
\end{cases}
\end{equation*}
With these values of $P$ and $Q$, Equation~\eqref{eq:lucas_series_pos} becomes
\begin{equation*}
    \sum_{n=2}^{\infty}\mathcal{L}\left(\frac{U^2_{k}}{U^2_{kn}}\right)  = \mathcal{L}\left(\frac{1}{\varphi^{2k}}\right).
\end{equation*}
If $k$ is odd, then we recover Example~\ref{ex:lucas_neg_fib} given above, which was originally obtained by setting $P=1$ and $Q = -1$.  Namely,
\begin{equation*}
    \sum_{n=1}^{\infty}\mathcal{L}\left(\frac{L_k}{5F^2_{2kn}}\right) + \sum_{n=1}^{\infty}\mathcal{L}\left(\frac{L_k}{L^2_{k(2n+1)}}\right)  = \mathcal{L}\left(\frac{1}{\varphi^{2k}}\right).
\end{equation*}
On the other hand, if $k$ is even then we find that
\begin{equation*}
    \sum_{n=2}^{\infty}\mathcal{L}\left(\frac{F^2_k}{F^2_{kn}}\right) = \mathcal{L}\left(\frac{1}{\varphi^{2k}}\right),
\end{equation*}
which is exactly the identity given in Example~\ref{ex:lucas_pos_fib}.
\end{example}
\section{Relation to Bridgeman's Theorem}
The sequences $(a_k)_{k \geq 1}$ and $(b_k)_{k \geq 1}$ appearing in Theorem~\ref{thm:bridgeman} are expressed in terms of powers of the quadratic irrational $u := a+b\sqrt{n}$, where $a,b \in \mathbb{Q}_{> 0}$.  In particular, they are defined as
\begin{equation*}
    a_k + b_k\sqrt{n} := (a+b\sqrt{n})^k.
\end{equation*}
We may express the sequences $(a_k)$ and $(b_k)$ in terms of the Lucas sequences $(U_k)$ and $(V_k)$ in the following way. Suppose that $P,Q \in \mathbb{Q}$ and that
\begin{equation*}
    \frac{P+\sqrt{P^2-4Q}}{2} = a+b\sqrt{n}.
\end{equation*}
Since $a,b \in \mathbb{Q}$, it must be the case that
\begin{equation*}
    P = 2a \quad \text{and} \quad \sqrt{4a^2-4Q} = 2b\sqrt{n}.
\end{equation*}
Thus, $Q$ is found to be
\begin{equation*}
    Q = a^2 - n b^2.
\end{equation*}
However, since Bridgeman's theorem requires that $(a,b)$ is a solution to one of Pell's equations
\begin{equation*}
    a^2 - n b^2 = \pm 1,
\end{equation*}
we find that $Q = \pm 1$.  Therefore, $P = 2a$ and $Q = \pm 1$, dependent on whether $(a,b)$ is a solution to the positive or negative Pell equation, respectively.  Next, recall that 
\begin{equation*}
    \left(\frac{P+\sqrt{P^2-4Q}}{2}\right)^k = \frac{V_k(P,Q) + U_k(P,Q)\sqrt{P^2-4Q}}{2}.
\end{equation*}
Comparing this with the definition of $a_k$ and $b_k$, we see that
\begin{equation*}
    V_k(2a,\pm 1)/2 + b U_k(2a,\pm 1) \sqrt{n} = a_k + b_k \sqrt{n},
\end{equation*}
and therefore,
\begin{equation*}
    a_k = \frac{1}{2}V_k(2a,\pm 1) \quad \text{and} \quad b_k = b U_k(2a,\pm 1).
\end{equation*}
Using these expressions for $a_k$ and $b_k$, Equation~\eqref{eq:bridgeman_pos} of Theorem~\ref{thm:bridgeman} may be rewritten as
\begin{equation}
    \mathcal{L}\left(\frac{1}{u^2}\right) = \sum_{k=2}^{\infty}\mathcal{L}\left(\frac{1}{U^2_k(2a,1)}\right), \label{eq:bridgeman-rewritten_a}
\end{equation}
while Equation~\eqref{eq:bridgeman_neg} may be expressed as
\begin{equation*}
    \mathcal{L}\left(\frac{1}{u^2}\right) = \sum_{k=1}^{\infty}\mathcal{L}\left(\frac{a^2}{n b^2 U^2_{2k}(2a,-1)}\right) + \sum_{k=1}^{\infty}\mathcal{L}\left(\frac{4a^2}{V^2_{2k+1}(2a,-1)}\right). 
\end{equation*}
Moreover, by using the fact that $V_1(2a,-1) = 2a$ and $D := P^2-4Q = 4b^2n$, the previous line may be written as
\begin{equation}
    \mathcal{L}\left(\frac{1}{u^2}\right) = \sum_{k=1}^{\infty}\mathcal{L}\left(\frac{V_1(2a,-1)}{D U^2_{2k}(2a,-1)}\right) + \sum_{k=1}^{\infty}\mathcal{L}\left(\frac{V_1(2a,-1)}{V^2_{2k+1}(2a,-1)}\right). \label{eq:bridgeman-rewritten_b}
\end{equation}
It is then easy to see that Equations~\eqref{eq:bridgeman-rewritten_a} and \eqref{eq:bridgeman-rewritten_b} correspond to Equations~\eqref{eq:lucas_series_pos} and \eqref{eq:lucas_series_neg} of Theorem~\ref{thm:generalized_bridgeman}, with $k = 1$, and $Q=\pm1$, respectively.

\vskip 20pt\noindent {\bf Acknowledgment.}
The author would like to thank the anonymous referee and the editors, whose suggestions and corrections improved the quality of this paper.

\end{document}